\documentclass{amsart}
\usepackage{delimset, amsrefs}
\usepackage[colorlinks]{hyperref}
\usepackage{tikz, tkz-berge}

\tikzset{OuterBoundary/.style={very thick}}
\tikzset{RimHook/.style={cyan, line width=1.2pt, rounded corners=2pt}}
\tikzset{edge/.style={cyan, very thick}}
\tikzset{ball/.style={shape=circle, inner sep=9pt, ball color=magenta!25!white}}
\tikzset{RHiball/.style={ball color=magenta!50!white}}
\tikzset{RHeball/.style={ball color=magenta!25!white}}
\tikzset{apball/.style={ball color=black}}

\SetVertexNoLabel
\tikzstyle{EdgeStyle}=[very thick, color=cyan]
\tikzstyle{VertexStyle}=[shape=circle, inner sep=6pt, ball color=magenta!25!white]

\usepackage{genyoungtabtikz}
\YFrench
\Yboxdim{1cm}

\newcommand\floor[1]{\lfloor#1\rfloor}
\newcommand\ceil[1]{\lceil#1\rceil}
\newcommand\Sq{\mathrm{Sq}}

\usepackage[capitalize]{cleveref}
\crefname{ineq}{Ineq.}{Ineqs.}

\newtheorem{theorem}{Theorem}[section]

\newtheorem{lemma}[theorem]{Lemma}
\newtheorem{conjecture}[theorem]{Conjecture}
\newtheorem{proposition}[theorem]{Proposition}

\theoremstyle{definition}
\newtheorem{example}[theorem]{Example}

\numberwithin{equation}{section}
\numberwithin{figure}{section}

\usepackage{xcolor}

\title[Non-Schur positivity of chromatic symmetric functions]
{Non-Schur positivity\\of chromatic symmetric functions}

\author[D.G.L. Wang]{David G.L. Wang$^\dag$$^\ddag$}
\address{
$^\dag$School of Mathematics and Statistics, Beijing Institute of 
Technology, 102488 Beijing, P. R. China\\
$^\ddag$Beijing Key Laboratory on MCAACI, Beijing Institute of 
Technology, 102488 Beijing, P. R. China}
\email{glw@bit.edu.cn; david.combin@gmail.com}

\author[M.M.Y. Wang]{Monica M.Y. Wang}
\address{
School of Mathematics and Statistics, Beijing Institute of 
Technology, 102488 Beijing, P. R. China}
\email{mengyu919@bit.edu.cn}



\thanks{Corresponding author: David G.L. Wang (glw@bit.edu.cn).
This paper was supported by General Program of National Natural 
Science Foundation of China (Grant No.\ 11671037).}

\begin{document}
\maketitle
\begin{abstract}
We provide a formula for every Schur coefficient in the chromatic symmetric function of a graph in terms of special rim hook tabloids. This formula is useful in confirming the non-Schur positivity of the chromatic symmetric function of a graph, especially when Stanley's stable partition method does not work. As applications, we determine Schur positive fan graphs and Schur positive complete tripartite graphs. We show that any squid graph obtained by adding $n$ leaves to a common vertex on an $m$-vertex cycle is not Schur positive if $m\ne 2n-1$, and conjecture that neither are the squid graphs with $m=2n-1$.
\end{abstract}

\tableofcontents

\section{Introduction}
Stanley~\cite{Sta95} introduced the {\it chromatic symmetric function}
for a simple graph $G$ as
\[
X_G=X_G(x_1, x_2, \ldots )=\sum_{\kappa}\prod_{v\in V(G)}x_{\kappa(v)}
\]
where the sum is over all proper colorings $\kappa$.
For example, the chromatic symmetric function of the complete graph $K_n$ 
is $n!e_n$, where $e_n=\prod_{i\ge 1}x_i$ is the $n$th elementary symmetric function.
It is a generalization of the chromatic polynomial $\chi_G(z)$ in the sense that $X_G(1^k)=\chi_G(k)$.
Shareshian and Wachs~\cite{SW16} refined
the concept of chromatic symmetric functions, which was called the
\emph{chromatic quasisymmetric function} of $G$, 
by considering all acyclic orientations of $G$ and
introducing another parameter $q$ to track the number of directed edges $(i,j)$
for which $i<j$. 
Ellzey~\cite{Ell17,Ell18D} further generalized these functions by
allowing orientations with directed cycles; see also
Alexandersson and Panova~\cite{AP18} 
for the same generalization from another perspective.

Let $\Lambda_n(x_1,x_2,\ldots)$ be the vector space of symmetric functions 
of degree $n$.
The Schur symmetric functions are considered to be 
the most important basis for $\Lambda_n$,
for its ubiquitousness in representation theory, 
mathematical physics and other areas,
and the crucial role it plays in understanding the representation theory 
of the symmetric group; see Macdonald~\cite{Mac95B,Mac15B}.
In fact,
every irreducible homogeneous polynomial representation $\phi$ of
the general linear group $GL_n(\mathbb{C})$ is given by
\[
\mathrm{char}(\phi)(x)=s_\lambda(x_1, x_2, \dots, x_n)
\]
for some partition $\lambda$ of $n$, 
where $s_\lambda(x_1, x_2, \dots, x_n)$ is a Schur polynomial.
Let $\mathfrak{S}_n$ be the symmetric group of order $n$.
Then for any $\mathfrak{S}_n$-module $M$,
\[
M=\bigoplus_{\lambda\vdash n}(S^{\lambda})^{\bigoplus c_\lambda} 
\quad\iff\quad
\mathrm{ch}_M(x)=\sum_{\lambda\vdash n}c_\lambda s_\lambda(x),
\]
where $S^\lambda$ are irreducible $\mathfrak{S}_n$-modules
and $\mathrm{ch}_M(x)$ is the Frobenius characteristic of~$M$; 
see Stanley~\cite{Sta99B} and Sagan~\cite{Sag01B}.

For any symmetric function basis $b_{\lambda}$, 
the graph $G$ is said to be {\it $b$-positive} 
if the expansion of $X_G$ in $b_{\lambda}$ has nonnegative coefficients. 
Note that $e$-positive graphs are Schur positive,
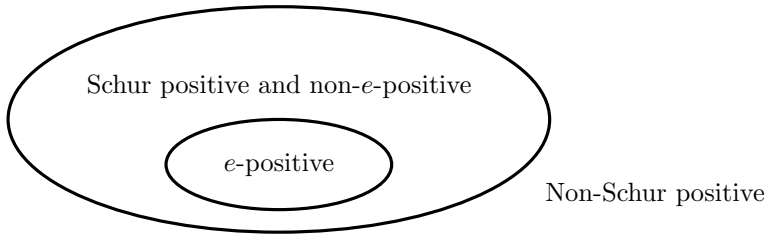
\begin{figure}[htbp]
\begin{tikzpicture}
\draw[OuterBoundary]
(0,0) ellipse [x radius=1.5cm, y radius=.6cm] node{$e$-positive};
\begin{scope}[yshift=.6cm]
\draw[OuterBoundary](0,0) ellipse [x radius=3.6cm, y radius=1.5cm]
node[above=4pt]{Schur positive and non-$e$-positive};
\draw(5,-1) node{Non-Schur positive};
\end{scope}
\end{tikzpicture}
\caption{The logical relation between $e$-positive symmetric functions 
and Schur positive symmetric functions.}\label{fig:e:s}
\end{figure}
since the coefficient of $s_\lambda$ in the Schur expansion of $e_\mu$ 
is the Kostka number $K_{\lambda',\,\mu}\ge0$,
where $\lambda$ and $\mu$ are partitions of $n$,
and $\lambda'$ is the conjugate of~$\lambda$;
see Mendes and Remmel~\cite[Theorem 2.22]{MR15B} and \cref{fig:e:s}.
The study of Schur positivity of chromatic symmetric functions 
is an active area due to its connections to the representation theory 
of the symmetric group and that of the general linear group. 
For instance, see Gasharov~\cite{Gas96P} and recent work
of Pawlowski~\cite{Paw18X}.

Many graph classes have been shown $e$-positive, including
complete graphs, paths, cycles, co-triangle-free graphs,
generalized bull graphs, 
(claw, $P_4$)-free graphs, 
(claw, paw)-free graphs, (claw, co-paw)-free graphs, 
(claw, triangle)-free graphs, (claw, co-$P_3$)-free graphs, $K$-chains,
lollipop graphs, triangular ladders;
see \cite{Sta95, CH18, FHM19, Tsu18, HHT19, Dah19X, DvW18}.
Schur positive graphs include
the incomparability graphs of (3+1)-free posets,
the incomparability graph of the natural unit interval order,
and the 2-edge-colorable hyperforests;
see \cite{Gas99, SW16, Paw18X}.
Non-$e$-positive graphs include the saltire graphs and
triangular tower graphs;
see \cite{DFvW17X, DSvW19X}.
We have not noticed any work concentrated on non-Schur positive graphs yet.

To show the non-$e$-positivity of a connected $n$-vertex graph,
Wolfgang~\cite[Proposition 1.3.3]{Wol97D} provided the following powerful tool.
\begin{proposition}[Wolfgang]\label{prop:ne:Wolfgang}
Every connected $e$-positive graph has a connected partition of any type.
\end{proposition}
For the non-Schur positivity, 
Stanley~\cite[Proposition 1.5]{Sta98} pointed the following tool.
\begin{proposition}[Stanley]\label{prop:nS:dominance}
Every Schur positive graph having a stable partition of type $\lambda$
has a stable partition of type $\mu$ for all partitions $\mu$ dominated by $\lambda$.
\end{proposition}

The smallest non-Schur positive connected graph is the claw, 
whose chromatic symmetric function is 
\[
X_{\text{claw}}=s_{31}-s_{2^2}+5s_{21^2}+8s_{1^4}.
\] 
All the other $4$-vertex connected graphs are Schur positive.

\Cref{prop:nS:dominance} is powerful in determining Schur positive 
wheel graphs, Schur positive windmill graphs, and Schur positive complete bipartite graphs; see \cref{eg:nS:wheel,eg:nS:windmill,eg:nS:Knm}, respectively.
However, it does not work
when a graph has a stable partition of some type $\lambda$
and a stable partition of any type $\mu$ that is dominated by $\lambda$. 
For instance, the fan graph $F_{4,6}$ is shown to be non-Schur positive but the type of every stable partition of $F_{4,6}$ is dominated with the type $\lambda=(4,3,3)$; see~\cref{thm:nS:fan}.

In this paper, we provide a combinatorial way to compute the Schur coefficients
of a chromatic symmetric function explicitly; see~\cref{thm:s-in-X}.
Our formula involves special rim hook tableaux 
and stable partitions of the graph 
whose type is the length list of the rim hooks.
We find it often not hard to compute 
some selected specific coefficients using the formula, 
especially when the graph is well structured, 
so that to determine the non-Schur positivity.

Using \cref{thm:s-in-X},
we determine Schur positive fan graphs,
Schur positive complete tripartite graphs,
and show the non-Schur positivity of a special kind of squid graphs.
We pose \cref{conj:sq:2n-1:1n} 
for further study.
The non-$e$-positivity of some of the above graphs has been confirmed 
by Dahlberg, She and van Willigenburg~\cite{DSvW19X}.
All results in this paper are checked by using Russell's program \cite{Mar20W}.

The chromatic symmetric functions $X_G$ 
was generalized to $Y_G$ in noncommuting variables by 
Gebhard and Sagan~\cite{GS01}. 
Dahlberg and van Willigenburg~\cite[Theorem~4.14]{DvW20}
showed that the $e$-positivity and Schur positivity results of 
the functions~$Y_G$ are simple and beautiful. In particular,
for any graph $G$ with distinct vertex labels in the set $[n]$,
the function $Y_G$ is $e$-positive if and only if 
$G$ is a disjoint union of complete graphs.

\section{Preliminaries}
A \emph{composition} of $n\ge 1$ is a list of integers that sum to $n$.
An \emph{integer partition} of~$n$ is a composition 
$\lambda=(\lambda_1,\lambda_2,\dots,\lambda_{\ell})$ of $n$
in non-increasing order, denoted $\lambda\vdash n$.
It can be written as
$\lambda=1^{m_1}2^{m_2}\cdots$
alternatively,
where $m_i$ is the multiplicity of~$i$ in~$\lambda$.
Denote $\lambda^!=\prod_{i\ge 1}m_i!$.
A \emph{stable partition} of a graph $G$ is a set partition
$\pi=\{V_1, V_2, \dots, V_m\}$
of the vertex set $V(G)$ such that every set $V_i$ is stable.
The cardinalities $\abs{V_1}$, $\abs{V_2}$, $\dots$, $\abs{V_m}$
form an integer partition, called the \emph{type} of $\pi$.
Besides the elementary symmetric functions $e_\lambda$
and Schur symmetric functions~$s_\lambda$, 
common bases for the ring $\Lambda_n$ include
the monomial symmetric functions~$m_\lambda$
and the power symmetric functions $p_\lambda$.
In terms of stable partitions, monomial and power symmetric functions,
Stanley~\cite{Sta95} obtained~\cref{prop:csf} as a basic way
to compute the chromatic symmetric function of a graph.

\begin{proposition}[Stanley]\label{prop:csf}
The chromatic symmetric function of a graph $G$ is
\[
X_G
=\sum_{\pi}\mu(\pi)^!\cdot m_{\mu(\pi)}
=\sum_{S\subseteq E(G)}(-1)^{\abs{S}}p_{\lambda(S)},
\]
where $\pi$ runs over stable partitions of $G$, 
and $\mu(\pi)$ is the type of $\pi$; $\lambda(S)$ 
is the integer partition of $\abs{V(G)}$ whose parts 
are the component orders of the graph~$(V(G),\,S)$.
\end{proposition}

Before introducing our main result \cref{thm:s-in-X}, 
we recall the power of Stanley's \cref{prop:nS:dominance} 
in proving the non-Schur positivity of a graph.
A partition $\mu=(\mu_1,\mu_2,\dots)$ is \emph{dominated} by
a partition $\lambda=(\lambda_1,\lambda_2,\dots)$, 
written as $\mu\trianglelefteq\lambda$, if 
$\sum_{i=1}^t\mu_i\le \sum_{i=1}^t\lambda_i$
for all $t$. The dominance ordering on the set of partitions is a partial order
and considered ``natural'' by Macdonald~\cite[Page 7]{Mac95B}.
In using \cref{prop:nS:dominance}, it suffices
to select two integer partitions $\lambda$ and $\mu$ such that 
\begin{itemize}
\item
$\mu$ is dominated by $\lambda$, 
\item
$G$ has a stable partition of type $\lambda$,
\item
$G$ has no stable partitions of type $\mu$.
\end{itemize}

\Cref{eg:nS:wheel,eg:nS:windmill} are applications of \cref{prop:nS:dominance} to show the non-Schur positivity.

\begin{example}\label{eg:nS:wheel}
The \emph{wheel graph} $W_n$ is the graph $C_{n-1}+P_1$
formed by connecting a single vertex to all vertices of the cycle $C_{n-1}$;
see \cref{fig:wheel8} for an illustration of~$W_8$.
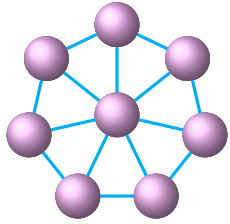
\begin{figure}[htbp]
\begin{tikzpicture}[rotate=90, scale=0.6]
\grWheel[RA=2]{8};
\end{tikzpicture}
\caption{The wheel graph $W_8$.}\label{fig:wheel8}
\end{figure}
The graph $W_4=K_4$ is $e$-positive,
and the graphs $W_5$ and $W_6$ are $e$-positive since
\[
X_{W_5}=70e_5+6e_{41}+2e_{32^2}
\quad\text{and}\quad
X_{W_6}=180e_6+40e_{51}+20e_{42}.
\]
When $n\ge 7$, the graph $W_n$ is not Schur positive.
This can be seen by using \cref{prop:nS:dominance} 
and taking $\mu\trianglelefteq\lambda$ 
in the following way.
If $n$ is odd, then we take 
\[
\lambda=\brk1{(n-1)/2,\,(n-1)/2,\,1}
\quad\text{and}\quad
\mu=\brk1{(n-1)/2,\,(n-3)/2,\,2}.
\]
If $n$ is even, then we take
\[
\lambda=(n/2-1,\,n/2-1,\,1,\,1)
\quad\text{and}\quad
\mu=(n/2-2,\,n/2-2,\,2,\,2).
\]
\end{example}

\begin{example}\label{eg:nS:windmill}
The {\it windmill graph} $W_n^d$ is obtained 
by joining $d$ copies of the complete graph~$K_n$
at a shared common vertex. \Cref{fig:windmill} illustrates the graph~$W_3^4$.
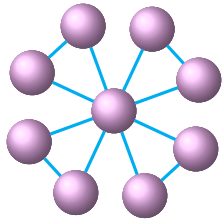
\begin{figure}[htbp]
\begin{tikzpicture}[rotate=65, scale=.6]
\grStar[prefix=u, RA=2]{9};
\EdgeInGraphMod*{u}{8}{1}{1}{2}
\end{tikzpicture}
\caption{The windmill graph $W_3^4$.}\label{fig:windmill}
\end{figure}
It is clear that $\abs1{W_n^d}=dn-d+1$.
Note that 
$W_1^d=K_1$ and
$W_n^1=K_n$ are $e$-positive.
Stanley~\cite[Corollary~3.6]{Sta95} proved that $W_n^2$ 
is $e$-positive by showing that any graph having a bipartition
consisting of two cliques is $e$-positive.
Dahlberg et~al.~\cite[Example~40]{DSvW19X} presented that
the star $W_2^d$ is not Schur positive for $d\ge 3$,
and that \cite[Example~36]{DSvW19X}
$W_n^d$ is not $e$-positive for $n,d\ge 3$.
In fact, for $n, d\ge 3$, the windmill graph $W_n^d$ is not Schur positive,
which can be seen by taking
\[
\lambda=\brk1{d^{n-1}1}
\quad\text{and}\quad
\mu=\brk1{d^{n-2}(d-1)2}
\]
and using \cref{prop:nS:dominance}.
\end{example}

\section{Non-Schur positivity of some graphs}\label{sec:result}

In order to state our main result \cref{thm:s-in-X},
we need more notions on Young tableaux. 
We follow terminologies from \cite{MR15B} and use the French notation;
see Fulton~\cite{Ful97B} for more information on Young tableaux.

The \emph{Young diagram} of an integer partition $\lambda$ is
the collection of left-justified rows of~$\lambda_i$ cells 
in the $i$th row reading from bottom to top.
A \emph{rim hook} is a sequence of connected cells in a Young diagram which
starts from a cell on the northeast boundary
and travels along the northeast edge
such that its removal leaves the Young diagram a smaller integer partition.
For any composition $\mu=(\mu_1,\mu_2,\dots,\mu_l)$ of $n$,
a \emph{rim hook tableau} of shape~$\lambda$ and content~$\mu$
is a filling of the cells of the Young diagram of $\lambda$
with rim hooks of lengths $\mu_l,\,\mu_{l-1},\,\dots$,
labeled with $1,\,2,\,\dots$, such that the removal of the last $i$ rim hooks
leaves the Young diagram of a smaller integer partition for all~$i$.
A \emph{rim hook tabloid} is a rim hook tableau with all rim hook labels removed.
A \emph{special rim hook tabloid} is a rim hook tabloid such that
every rim hook intersects the first column of the tabloid.
Let $\mathcal{T}(\lambda)$ be the set of special rim hook tabloids 
of shape~$\lambda$; see \cref{fig:SRHT} for illustration.
\begin{figure}[htbp]
\centering
\begin{tikzpicture}
\tyng(0cm, 0cm, 5^2,4,3,1)
\coordinate (S1) at (0.5, 0.5);
\coordinate (S2) at (.5, 1.5);
\coordinate (S3) at (.5, 2.5); 
\coordinate (S4) at (.5, 4.5);
\coordinate (1) at (3.5, 0.5); 
\coordinate (2) at (4.5, 0.5);
\coordinate (3) at (1.5, 2.5); 
\coordinate (4) at (3.5, 2.5);
\draw[RimHook]
(S4) -- (.5, 3.5) -- (2.5, 3.5) -- (2.5, 2.5) -- (4) 
(S3) -- (3) 
(S2) -- (4.5, 1.5) -- (2)
(S1) -- (1);
\draw[OuterBoundary]
(0, 0) 
-- (5, 0)
-- (5, 2)
-- (4, 2)
-- (4, 3)
-- (3, 3)
-- (3, 4)
-- (1, 4)
-- (1, 5)
-- (0, 5)
-- (0, -.025);
\foreach \s in {S1, S2, S3, S4}
  \shade[RHiball](\s) circle(.1);
\foreach \e in {1, 2, 3, 4}
  \shade[RHeball](\e) circle(.3) node{\e};

\begin{scope}[xshift=6cm]
\tyng(0cm, 0cm, 5^2,4,3,1)
\coordinate (S1) at (.5, 0.5);
\coordinate (S2) at (.5, 1.5);
\coordinate (S3) at (.5, 2.5); 
\coordinate (S4) at (.5, 4.5);
\coordinate (1)  at (3.5, 0.5); 
\coordinate (2)  at (4.5, 0.5);
\coordinate (3)  at (1.5, 2.5); 
\coordinate (4)  at (3.5, 2.5);
\draw[RimHook]
(S4) -- (.5, 3.5) -- (2.5, 3.5) -- (2.5, 2.5) -- (4) 
(S3) -- (3) 
(S2) -- (4.5, 1.5) -- (2)
(S1) -- (1);
\draw[OuterBoundary]
(0,0)
-- (5, 0)
-- (5, 2)
-- (4, 2)
-- (4, 3)
-- (3, 3)
-- (3, 4)
-- (1, 4)
-- (1, 5)
-- (0, 5)
-- (0, -.025);
\foreach \s in {S1, S2, S3, S4, 1, 2, 3, 4}
  \shade[RHiball](\s) circle(.1);
\end{scope}
\end{tikzpicture}
\caption{A rim hook tableau of shape $\lambda=(5,5,4,3,1)$ 
and content $\mu=(4,6,2,6)$; 
a special rim hook tabloid of shape $\lambda$ and content $\mu$.}
\label{fig:SRHT}
\end{figure}
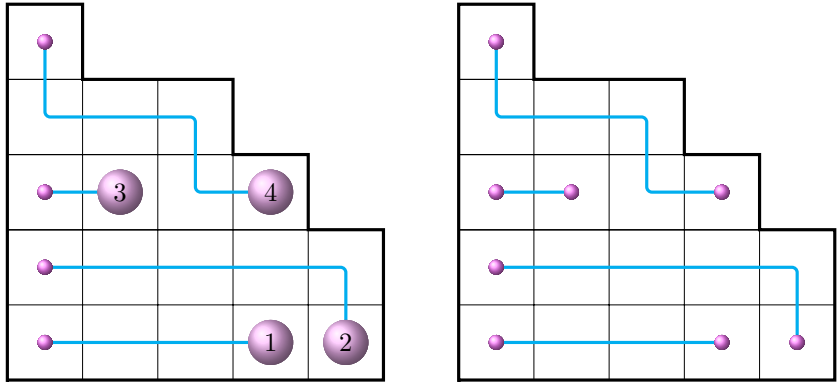
For any $T\in\mathcal{T}(\lambda)$, denote by
$\pi(T)=(\pi_1,\,\pi_2,\,\ldots)$
the length list of its rim hooks from up to bottom,
and by $W(T)$ the set of row labels $i$ of $T$ from up to bottom
such that the rim hook with an end in Row $i$ and Column~$1$
spans an even number of rows.
Since any composition~$\tau$ of $n$ determines 
at most one tabloid $T\in\mathcal{T}(\lambda)$
such that $\pi(T)=\tau$,
one may identify a tabloid in $\mathcal{T}(\lambda)$ by clarifying $\pi(T)$.

Here is our first main result.

\begin{theorem}\label{thm:s-in-X}
For any graph $G$ and any integer partition $\lambda\vdash \abs{V(G)}$,
\[
X_G
=\sum_{T\in\mathcal{T}(\lambda)}(-1)^{\abs{W(T)}}\ \pi(T)^!\ N_G(T)\cdot s_\lambda,
\]
where $\mathcal{T}(\lambda)$ is the set of special rim hook tabloids 
of shape $\lambda$;
$\abs{W(T)}$ is the number of rim hooks that span an even number of rows;
$\pi(T)$ is the integer partition of~$\abs{V(G)}$ 
whose parts are the rim hook lengths of $T$;
and $N_G(T)$ is the number of stable partitions of $G$ whose type is $\pi(T)$.
\end{theorem}
\begin{proof}
Recall from~\cite[Exercise 2.15]{MR15B} that 
the coefficient of $s_\lambda$ in the monomial symmetric function~$m_\mu$ equals
the inverse Kostka number
\[
K_{\mu,\lambda}^{-1}
=\sum_{T\in\mathcal{T}(\lambda,\mu)}\prod_{H}(-1)^{r(H)-1},
\]
where $\mathcal{T}(\lambda,\mu)$ is the set 
of special rim hook tabloids of shape $\lambda$ and content $\mu$,
$H$ runs over rim hooks of $T$, and $r(H)$ is the number of rows that $H$ spans.
This result simplifies to
\[
m_\mu
=\sum_{T\in\mathcal{T}(\lambda,\mu)}(-1)^{\abs{W(T)}}s_\lambda,
\]
in virtue of the definition of $W(T)$. 
Now, we use the notation $[s_\lambda]f$ 
to denote the coefficient of $s_\lambda$ in the Schur expansion of a symmetric function $f$. 
By \cref{prop:csf}, we can deduce that
\begin{align*}
[s_\lambda]X_G
&=[s_\lambda]\sum_{\pi}\mu(\pi)^!\ m_{\mu(\pi)}\\
&=\sum_{\pi}\mu(\pi)^!\ \sum_{T\in\mathcal{T}
\brk1{\lambda,\,\mu(\pi)}}(-1)^{\abs{W(T)}}\\
&=\sum_{T\in\mathcal{T}(\lambda)}(-1)^{\abs{W(T)}}\cdot \pi(T)^!\cdot N_G(T),
\end{align*}
where $\pi$ runs over stable partitions of $G$,
and $\mu(\pi)$ is the type of $\pi$.
\end{proof}

\Cref{thm:s-in-X} not only helps determine the non-Schur positivity 
of a graph $G$, but also gives the precise value 
of the coefficient of $s_\lambda$ in $X_G$.
In the remaining of \cref{sec:result},
we show the non-Schur positivity for graphs
in some popular graph families.
The general idea is to discover a particular shape $\lambda$
and show that the sum in \cref{thm:s-in-X} for this $\lambda$ is negative.
For the purpose of showing the negativity of a term,
we often ignore the value of the factor $\pi^!$
and use the fact of the positivity of $\pi^!$ only.
In fact, the product $\pi(T)^!N_G(T)$ is the number of stable partitions $S$
of~$G$ of type $\pi(T)$ such that the stable sets in $S$ with the same number $x$
of vertices are ordered for any $x$.

When the graph $G$ under consideration is clear from context,
we use the notation~$c_\lambda'$ to denote the coefficient $[s_\lambda]X_G$,
for the notation~$c_\lambda$ 
is widely used to denote the coefficient of $e_\lambda$ 
in the $e$-expansion of $X_G$.
In applications of \cref{thm:s-in-X},
we often shorten the notation $\pi(T)$, $W(T)$ and $N_G(T)$,
when no confusion arises, as~$\pi$, $W$ and $N$, respectively. 
\Cref{lem:tabloid} will be useful in the sequel.

\begin{lemma}\label{lem:tabloid}
Let $G$ be a graph and let $\lambda$ be a partition of $\abs{V(G)}$.
For any tabloid $T\in\mathcal{T}(\lambda)$ with $N_G(T)>0$, we have the following.
\begin{itemize}
\item
Every part $\pi_i$ in $\pi(T)$ is less than or equal to 
the independent number $\alpha(G)$.
\item
The length of $\pi(T)$ is less than or equal to the length of $\lambda$.
\item
If every stable partition of $G$ contains a singleton stable set,
then some part~$\pi_i$ in $\pi(T)$ equals $1$.
\end{itemize}
\end{lemma}
\begin{proof}
Let $T\in\mathcal{T}(\lambda)$ with $N_G(T)>0$.
Since every stable set of $G$ contains at most $\alpha(G)$ vertices,
no rim hook in $T$ is longer than $\alpha(G)$.
Since every rim hook intersects the first column of
the Young diagram of $\lambda$, we find that $\pi$ is not longer than $\lambda$.
The last statement holds true since the type of a stable partition of $G$ is the content of $T$.
\end{proof}

A partition $\pi$ of a graph $G$ is 
a set partition $\pi=\{V_1,V_2,\ldots,V_k\}$ of $V(G)$.
We call the integer partition obtained by rearranging the numbers in the sequence 
$\brk1{\abs{V_1},\,\abs{V_2},\,\ldots,\,\abs{V_k}}$
the \emph{type} of $\pi$.
The partition $\pi$ is a \emph{bipartition} when $k=2$,
and a \emph{tripartition} when $k=3$.
It is \emph{balanced} if 
\[
\max\brk[c]1{\abs{V_j}\colon 1\le j\le k}
\le \min\brk[c]1{\abs{V_j}\colon 1\le j\le k}+1.
\]

\subsection{Non-Schur positivity of connected bipartite graphs}

Dahlberg et al.~\cite[Theorem~39]{DSvW19X} proved that 
any bipartite $n$-vertex graph with a vertex of degree greater than $\ceil{n/2}$
is not Schur positive by using \cref{prop:nS:dominance}. 
For connected bipartite graphs, 
\cref{prop:nS:dominance} implies the following stronger result.

\begin{theorem}\label{thm:nS:bipartite}
Any Schur positive connected bipartite graph
has a balanced stable bipartition, namely,
the cardinalities of the two parts differ by $1$ or $0$.
\end{theorem}
\begin{proof}
Let $G=(U,V)$ be a connected bipartite graph.
Since $G$ is connected, the bipartition $(U,V)$ is unique.
Assume that $\abs{U}-\abs{V}\ge 2$. 
By \cref{prop:nS:dominance},
$G$ has a stable bipartition $(X,Y)$ such that $\abs{X}=\abs{U}-1$
and $\abs{Y}=\abs{V}+1$.
It is clear that $\{X,Y\}\ne\{U,V\}$, contradicting to the uniqueness of $(U,V)$.
\end{proof}

\Cref{thm:nS:bipartite} implies immediately that 
every spider with at least 3 legs of odd length is not Schur positive,
while the non-$e$-positivity appeared in 
Dahlberg et~al.~\cite[Corollary~16]{DSvW19X}.
In fact, spider graphs are bipartite.
The cardinality difference between the parts of a spider graph 
is one less than the number of legs of odd length.
As a consequence, every spider with at least 3 legs of odd length
is unbalanced and not Schur positive.

The converse of \cref{thm:nS:bipartite} is not true.
This can be seen from \cref{eg:nS:Knm}.

\begin{example}\label{eg:nS:Knm}
The complete bipartite graphs 
$K_{1,1}$, $K_{1,2}$, $K_{2,2}$, and $K_{2,3}$
are $e$-positive,
and all the other complete bipartite graphs are not Schur positive.
In particular, any star of at least $4$ vertices is not Schur positive.
\end{example}
\begin{proof}
Let $K_{m,n}$ be a complete bipartite graph. 
By \cref{thm:nS:bipartite}, we can suppose that $n\in\{m,\,m+1\}$.
Recall that Stanley~\cite[Propositions 5.3 and 5.4]{Sta95} 
showed that all paths and cycles are $e$-positive.
Thus the paths $K_{1,1}$ and $K_{1,2}$, 
and the cycle $K_{2,2}$ are $e$-positive.
The graph $K_{2,3}$ is $e$-positive since
\[
X_{K_{2,3}}=e_{2^21}+9e_{41}+e_{32}+35e_5.
\]
Let $m\ge 3$ and $\lambda=(n,m)$.
When $n=m$, one may obtain the non-Schur positivity by using \cref{prop:nS:dominance} and considering $\mu=(m-1,\,m-1,\,2)$.
When $n=m+1\ge 5$,  one may obtain the non-Schur positivity by using \cref{prop:nS:dominance} and considering
$\mu=(m-1,\,m-1,\,3)$.
The graph $K_{3,4}$ is not Schur positive since
\begin{multline*}
X_{K_{3,4}}
=s_{43}+2s_{421}+s_{41^3}+4s_{3^21}+
20s_{321^2}+32s_{31^4}\\+70s_{2^21^3}+292s_{21^5}+1066s_{1^6}
-4s_{32^2}-3s_{2^31}.
\end{multline*}
This completes the proof.
\end{proof}

\subsection{Non-Schur positive tripartite graphs}
Using \cref{thm:s-in-X}, we are able to 
determine all Schur-positive complete tripartite graphs.

\begin{lemma}\label{lem:ns-pos:tripartiteG}
Let $G$ be a Schur positive tripartite graph.
If all stable tripartitions of $G$ have the same type, then
any stable tripartition of $G$ is balanced. 
\end{lemma}
\begin{proof}
Immediate from \cref{prop:nS:dominance}.
\end{proof}

The \emph{fan} graph $F_{m, n}$ is defined to be 
the join $\overline{K_m}+P_n$ of the complement 
of the complete graph $K_m$ and the path $P_n$.

\begin{theorem}\label{thm:nS:fan}
The fan graph $F_{m,n}$ is Schur positive if and only if 
\[
(m,n)\in\brk[c]1{(1,n)\colon n\in[4]}\cup\brk[c]1{(2,n)\colon n\in[6]}\cup\{(3,4)\}.
\]
\end{theorem}
\begin{proof}
The fan graph $G=F_{m,n}$ has a unique tripartition, whose type is 
\[
\lambda=(m,\,\floor{n/2},\,\ceil{n/2}).
\]
By \cref{lem:ns-pos:tripartiteG}, any two of these 3 integers
differ at most $1$, i.e., $2m-2\le n\le 2m+2$.
Consider the partition
\[
\mu=\begin{cases}
(m+1,\,m-1,\,m-1,\,3),&\text{if $n=2m+2$};\\
(m+1,\,m-1,\,m-1,\,2),&\text{if $n=2m+1$};\\
(m-1,\,m-1,\,m-1,\,3),&\text{if $n=2m$};\\
(m-1,\,m-1,\,m-2,\,3),&\text{if $n=2m-1$};\\
(m-1,\,m-2,\,m-2,\,3),&\text{if $n=2m-2$}.
\end{cases}
\]
It is routine to check that $\mu\trianglelefteq\lambda$.
By \cref{prop:nS:dominance}, $G$ has a stable partition of type $\mu$.
Since any stable partition of $G$ contains a stable partition of $\overline{K_m}$, we infer that some parts in $\mu$ sum to $m$.
It follows that $m\le 4$.
Moreover, if $m=4$, then $n=2m-2=6$; if $m=3$, then $n\le 2m-1=5$.
The remaining of this proof can be done by computer calculation. 
We list the results below for completeness.

The non-Schur positivity of the graph $F_{4,6}$ can be seen by computing the coefficient $c_\tau'$ where $\tau=(3^22^2)$ using \cref{thm:s-in-X}. In fact,
there are 5 tabloids in $\mathcal{T}(\tau)$ with $N>0$:
\begin{enumerate}
\item
$\pi=(3,4,3)$, $W=\{1,2\}$, $\pi^!=2$ and $N=1$.
\item
$\pi=(3,1,4,2)$ and $W=\{1,3\}$.
\item
$\pi=(3,1,3,3)$, $W=\{1\}$ and $\pi^!=6$.
\item
$\pi=(2,4,1,3)$ and $W=\{2\}$.
\item
$\pi=(2,2,4,2)$ and $W=\{3\}$.
\end{enumerate}
By \cref{thm:s-in-X},
only the first two tabloids have positive contributions to the coefficient $c_\tau'$.
The positive contribution of the second tabloid equals the negative contribution of the fourth tabloid, since they have the same partition $\lambda(\pi)=(4,3,2,1)$.
The first tabloid contributes positive $2$ to $c_\tau'$,
and the third tabloid contributes negative $\pi^!N\le -6$.
Hence $c_\tau'\le 2-6$ is negative. 

When $m\le 2$ and $2m-2\le n\le 2m+2$, 
the complete graphs $F_{1,1}=K_2$ and $F_{1,2}=K_3$ are $e$-positive,
the graphs $F_{1,3}$ and $F_{2,2}$ are isomorphic to $K_4-e$ and
\begin{align*}
X_{F_{1,3}}&=X_{F_{2,2}}=X_{K_4-e}=16e_4+2e_{31},\\
X_{F_{1,4}}&=40e_5+12e_{41}+2e_{32},\\
X_{F_{2,3}}&=70e_5+6e_{41}+2e_{32},\\
X_{F_{2,4}}&=276e_6+44e_{51}+4e_{42}+6e_{3^2},\\
X_{F_{2,5}}&=1022e_7+298e_{61}+18e_{52}+12e_{51^2}+22e_{43}+2e_{3^21},\\
X_{F_{2,6}}&=2s_{3^22}+2s_{3^21^2}+14s_{32^21}+44s_{321^3}+212s_{31^5}+68s_{2^4}\\
&\qquad+50s_{2^31^2}+410s_{2^21^4}+2238s_{21^6}+5658s_{1^8},\\
X_{F_{3,4}}
&=2s_{32^2}+4s_{321^2}+8s_{31^4}+4s_{2^31}+70s_{2^21^3}+300s_{21^5}+1902s_{1^7},\\
X_{F_{3,5}}&=2s_{3^22}+2s_{3^21^2}+12s_{32^21}+24s_{321^3}+62s_{31^5}\\
&\qquad-46s_{2^4}+120s_{2^31^2}+428s_{2^21^4}+2088s_{21^6}+10554s_{1^8}.
\end{align*}
This completes the proof.
\end{proof}
From the above proof we see that there are 
only two Schur positive fan graphs $F_{m,n}$
which are not $e$-positive, i.e., $F_{2,6}$ and $F_{3,4}$:
\begin{align*}
X_{F_{2,6}}&=3632e_8+1660e_{71}+160e_{62}+170e_{61^2}\\
&\qquad-62e_{53}+30e_{521}+56e_{4^2}+10e_{431}+2e_{3^22},\\
X_{F_{3,4}}&=1610e_7+226e_{61}+60e_{52}+4e_{51^2}-2e_{43}+2_{421}+2e_{3^21}.
\end{align*}

\begin{theorem}\label{thm:nS:tripartite}
The graphs
$K_{1,1,1}$, $K_{1,1,2}$, $K_{1,2,2}$, $K_{2,2,2}$ and $K_{2,2,3}$
are $e$-positive, and all the other complete tripartite graphs are not Schur positive.
\end{theorem}
\begin{proof}
The $e$-positivity of the 5 graphs can be seen from the following:
\begin{align*}
X_{K_{1,1,1}}&=X_{K_3}=6e_3,\\
X_{K_{1,1,2}}&=X_{K_4-e}=16e_4+2e_{31},\\
X_{K_{1,2,2}}&=6e_{41}+2e_{32}+70e_5,\\
X_{K_{2,2,2}}&=36e_{51}+6e_{3^2}+384e_6,\quad\text{and}\quad\\
X_{K_{2,2,3}}&=12e_{51^2}+2e_{3^21}+268e_{61}+12e_{52}+4e_{43}+1988e_7.
\end{align*}
Let $G=K_{r,s,t}$ ($1\le r\le s\le t$) be a complete tripartite graph
which is not one of the above $6$ graphs. 
By \cref{lem:ns-pos:tripartiteG}, we can suppose that $t-r\le 1$.
The non-Schur positivity of $K_{2,3,3}$ can be seen 
by taking $\lambda=(3^22)$ and $\mu=(2^4)$ using \cref{prop:nS:dominance}.

Consider the tabloids $T\in\mathcal{T}_{G}(\lambda)$, where
\[
\lambda=(t,\,s-1,\,r-1,\,2)
\]
Besides \cref{lem:tabloid,thm:s-in-X}, we will use the fact that 
\[
N_{G}(T)=1\qquad\text{if $\pi(T)$ is a rearrangement of $r$, $s$ and $t$}.
\]
Since $t-r\le 1$, we have $3$ cases to treat.
\begin{enumerate}
\item
If $s=r=t$, then the only tabloid with even $\abs{W}$ satisfies
\[
\pi=(s,s,s),\quad
W=\{1,2\},\quad\text{and}\quad
\pi^!=6.
\]
One of the other tabloids satisfies
\[
\pi=(s,\,1,\,s-1,\,s),\quad
W=\{1\},\quad
\pi^!=2,\quad\text{and}\quad
N=3s.
\]
Therefore, we obtain $c_\lambda'\le 6-6s<0$.
\item
If $s=r=t-1$, then the only tabloid with even $\abs{W}$ satisfies
\[
\pi=(s,\,s,\,s+1),\quad
W=\{1,2\},\quad\text{and}\quad
\pi^!=2.
\]
On the other hand, one of the other tabloids satisfies
\[
\pi=(s,\,1,\,s-1,\,s+1),\quad
W=\{1\},\quad
\pi^!=1,\quad\text{and}\quad
N=2s.
\]
Therefore, we obtain $c_\lambda'\le 2-2s<0$.
\item
If $s=r+1=t$, then $s\ge 4$, and there are $4$ tabloids in all:
\begin{itemize}
\item
$\pi=(2,\,s-2,\,s-1,\,s)$, $W=\emptyset$, and $\pi^!N=2\binom{s}{2}$.
\item
$\pi=(s-1,\,1,\,s-1,\,s)$, $W=\{1\}$, $\pi^!=2$, and $N=2s$.
\item
$\pi=(2,\,s,\,s-3,\,s)$, $W=\{2\}$, $\pi^!\ge 2$, and  $N=\binom{s-1}{2}$.
\item
$\pi=(s-1,\,s,\,s)$, $W=\{1,2\}$, and $\pi^!=2$.
\end{itemize}
Therefore, we obtain
\[
c_\lambda'\le 2\binom{s}{2}-2s-2\binom{s-1}{2}+2=-2s<0.
\]
\end{enumerate}
This completes the proof.
\end{proof}

\subsection{Non-Schur positivity of squid graphs with unit legs}

A graph is a \emph{squid} if it is connected, unicyclic, 
and has only one vertex of degree greater than~2. 
It is named by Martin, Morin and Wagner~\cite{MMW08}, denoted $\Sq(m;\lambda)$,
where $m\geq 2$,
$\lambda=(\lambda_1,\lambda_2,\dots,\lambda_h)$,
and defined alternatively as a collection of edge-disjoint paths of lengths $\lambda_1,\dots, \lambda_h$, 
respectively, joined at a common endpoint on an $m$-cycle.

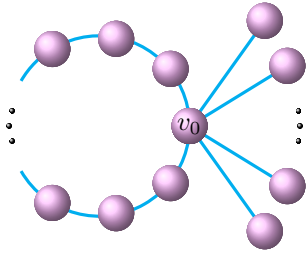
\begin{figure}[htbp]
\begin{tikzpicture}[scale=0.4, rotate=-90]
\coordinate (-3) at (-2.56, -1.56);
\coordinate (-2) at (-2.9, 0.56);
\coordinate (-1) at (-1.9, 2.37);
\coordinate (0)  at (0, 3);
\coordinate (1)  at (1.9, 2.37);
\coordinate (2)  at (2.9, 0.56);
\coordinate (3)  at (2.56, -1.56);
\coordinate (a)  at (-.5, -2.9);
\coordinate (b)  at (0, -3); 
\coordinate (c)  at (.5, -2.9);
\coordinate (4)  at (-3.5, 5+.5);
\coordinate (5)  at (-2, 5.75+.5);
\coordinate (6)  at (2, 5.75+.5);
\coordinate (7)  at (3.5, 5+.5);
\coordinate (d)  at (-.5, 6.12+.5);
\coordinate (e)  at (0, 6.17+.5); 
\coordinate (f)  at (.5, 6.12+.5);
\draw[edge](4)--(0)--(5) (6)--(0)--(7);   
\draw[edge] (1.5, -2.6) arc [start angle=-60, end angle=240, radius=3cm];
\foreach \e in {-3,...,7}\shade[ball](\e) circle(.6);
\shade[ball](0)circle(.6)node{$v_0$};
\foreach \e in {a, b, c, d, e, f}\shade[apball](\e) circle(.08);
\end{tikzpicture}
\caption{The graph $\Sq(m;1^n)$ with $m\ge 4$ and $n\ge 2$.}
\label{fig:Sq:m:1n}
\end{figure}

\begin{theorem}\label{thm:sq:m:1n}
Let $m\ge 3$, $n\ge 2$, and $m\ne 2n-1$.
Then the squid graph $\Sq(m;1^n)$ is not Schur positive.
\end{theorem}
\begin{proof}
Let $G=\Sq(m;1^n)$. Then $G$ has independence number $\alpha=n+\floor{m/2}$.
Let~$v_0$ be the vertex on the cycle that has additional leaves; see \cref{fig:Sq:m:1n}.
If $m$ is even, then $G$ is bipartite and unbalanced. 
By \cref{thm:nS:bipartite}, $G$ is not Schur positive.
Below we can suppose that $m$ is odd. 

If $m\le 2n-3$, then $n-1\ge \ceil{m/2}$ and $\mu\trianglelefteq\lambda$, where
\[
\lambda=\brk1{\alpha,\,\floor{m/2},\,1}
\quad\text{and}\quad
\mu=\brk1{n-1,\,\ceil{m/2},\,\ceil{m/2}}.
\]
Since any stable set containing $v_0$ has at most $\floor{m/2}$ vertices,
$G$ has no stable partitions of type $\mu$. By \cref{prop:nS:dominance}, $G$ is not Schur positive.

When $m\ge 2n+1$, we consider the partition
\[
\lambda=\brk1{\alpha-1,\,\ceil{m/2},\,1}.
\]
Since $G$ is not bipartite, 
in using \cref{thm:s-in-X}, we are restricted to 
tabloids containing at least 3 rim hooks.
By \cref{lem:tabloid}, we find
\[
\pi=\brk1{1,\,\ceil{m/2},\,\ceil{m/2}}\quad\text{or}\quad
\pi=\brk1{1,\,\ceil{m/2}+1,\,\floor{m/2}}.
\]
By \cref{thm:s-in-X}, we obtain 
\[
c_\lambda'=2\binom{n}{n-1}-2\cdot\frac{m+1}{2}=2n-m-1<0.
\]
This completes the proof.
\end{proof}

Dahlberg et al.~\cite{DSvW19X} showed that any $n$-vertex connected graph with a cut vertex whose removal produces at least 
$\max(3,\,\log_2{n}+1)$ components is not $e$-positive.
Using it one obtains immediately that the squid graph $\Sq(2n-1;\,1^n)$ is not $e$-positive. 

\begin{conjecture}\label{conj:sq:2n-1:1n}
The squid graph $\Sq(2n-1;\,1^n)$ is not Schur positive for any $n\ge 3$.
\end{conjecture}

The unique term in the chromatic symmetric function of the graph $\Sq(5;1^3)$ 
with negative coefficient is $-8s_{3^22}$.
The chromatic symmetric function of the graph $\Sq(7;1^4)$ has only two terms with negative coefficients, i.e., $-60s_{4^23}-30s_{3^32}$.
Careful and laborious enumeration with the aid of \cref{thm:s-in-X} gives the coefficient~$c_\lambda'$ for $\lambda=(n,\,n,\,n-1)$ by
\[
c_\lambda'=\frac{P(n)}{30n(n+1)(2n-7)(2n-5)(2n-3)}\binom{2n-2}{n-1},
\]
where $P(n)$ is the following polynomial which is not factorizable over $\mathbb{Z}$:
\[
P(n)=n^9-7n^8-96n^7+912n^6-2646n^5+4782n^4-17809n^3+59113n^2-91050n+50400.
\]
Hence $c_\lambda'<0$ and \cref{conj:sq:2n-1:1n} is true for $n\le 9$.

\end{document}